\documentclass[11pt]{article}

\usepackage{amsmath,amssymb,amsthm}
\usepackage{graphicx}
\usepackage{array}
\usepackage{booktabs}
\usepackage[margin=1in]{geometry}
\usepackage{hyperref}

\newtheorem{theorem}{Theorem}
\newtheorem{lemma}{Lemma}
\newtheorem{definition}{Definition}
\newtheorem{fact}{Fact}

\newcommand{\N}{\mathbb{N}}
\newcommand{\R}{\mathbb{R}}
\newcommand{\G}{\mathcal{G}}

\title{Monotone Property Thresholds\thanks{This paper was originally written by the authors circa 2005 but was never submitted for publication. The present version corrects minor errors, adds references to work published since the original draft, and includes a section discussing further research directions. The core framework---the property function $\psi_G$, the density function $\varphi^\psi_\mu$, and all threshold results---is unchanged from the original.}}
\author{Colton Magnant\thanks{Department of Mathematical Sciences, Georgia Southern University, Statesboro, GA 30460, USA.}\quad and \quad Thor Whalen\thanks{Independent researcher.}}
\date{Originally written circa 2005; this version March 2026}

\begin{document}
\maketitle

\begin{center}
\end{center}

\begin{abstract}
Many graph vulnerability parameters --- connectivity, toughness, scattering number, integrity, component-order connectivity --- have been studied independently, each with its own threshold results. We introduce a \emph{property function} $\psi_G$ that encodes all of these parameters as geometric properties of a single integer-valued function, and a \emph{density function} $\varphi^\psi_\mu$ that converts threshold problems into optimization over forbidden regions. This reduction yields a uniform derivation of known thresholds and produces new ones --- including the first minimum-degree thresholds for component-order connectivity --- as special cases of a single master theorem. We conclude with directions for extending the framework to spectral density parameters, edge-removal analogues, and directed graphs.
\end{abstract}

\section{Introduction}

A graph vulnerability parameter measures how connected a graph is---or more precisely, how disconnected the graph can become if we remove vertices of this graph\footnote{One must note that most said ``vulnerability'' parameters in fact measure vulnerability negatively in that a higher value of the parameter indicates a less vulnerable graph, as opposed to a more vulnerable one.}. In a communications network, one uses vulnerability measures to assess its resistance to disruption of operation in case of the failure of certain stations (vertices) or communication links (edges).

In all the following, $G$ will denote a simple graph and $n$ the order of this graph.

Connectivity is a fundamental concept in graph theory, appearing in premises and conclusions of many results. A graph $G$ is said to be \emph{connected} if any two vertices are connected by a path in $G$. From this simple definition of connectivity, many other graph properties/parameters were created to measure how close a disconnected graph is to being connected, and/or how far a connected graph is to being disconnected (i.e.\ how strongly connected a graph is)\footnote{These all measure (dis-)connectivity in different ways, but any sensible connectivity measures should agree on at least two things; no graph should be more disconnected than the empty graph (the graph with no edges) or more connected than the complete graph (the graph with all possible edges).}.

The following illustration will be useful when considering measures of graph vulnerability. Suppose we are the designers of a network and our goal is to make sure the network (i.e.\ graph) remains as connected as possible in the presence of a malicious enemy who knows the exact structure of the network and whose goal is to make the network as disconnected as possible by removing a certain number of nodes (i.e.\ vertices). The question arises as to what ``more disconnected'' means.

Since a disconnected graph comprises several connected ``pieces'' (called components---formally, a maximal connected subgraph), one usually measures how disconnected a graph is according to the number and order of these components. Consider the following two graph parameters:

\begin{definition}
Let $\omega(G)$ be the number of components and $\Omega(G)$ be the order of the largest component of~$G$.
\end{definition}

It is natural to think of $\omega(G)$ as positively measuring the ``disconnectedness'' of~$G$. That is, if $\omega(G) > \omega(H)$, $G$ could understandably be construed as being more disconnected (i.e.\ further from being connected) than $H$ in the sense that $G$ is ``broken into more pieces''.

It is also fairly natural to think of $\Omega(G)$ as negatively measuring the ``disconnectedness'' of~$G$. That is, if $\Omega(G) < \Omega(H)$, $G$ could understandably be construed as being more disconnected than~$H$ since $H$ enjoys a ``bigger connected piece'' than $G$ has.

When a graph is connected, a standard way of measuring the strength of this connectivity is to see how many vertices one would have to remove to disconnect the graph. A set of vertices whose removal disconnects a graph is called a \emph{cutset} of the graph. The connectivity $\kappa(G)$ is defined to be the minimum order of a cutset of~$G$, and $G$ is said to be \emph{$k$-connected} if $\kappa(G) \ge k$.

Consider the following simple Fact:

\begin{fact}\label{fact:basic}
\begin{align*}
G \text{ is connected} &\iff \omega(G) = 1 \\
G \text{ is disconnected} &\iff \omega(G) \in \{2, \ldots, n\} \\
G \text{ is connected} &\iff \Omega(G) = n \\
G \text{ is disconnected} &\iff \Omega(G) \in \{1, \ldots, n-1\} \\
G \text{ is connected} &\iff \kappa(G) \in \{1, \ldots, n-1\} \\
G \text{ is disconnected} &\iff \kappa(G) = 0
\end{align*}
\end{fact}

As we see, though $\kappa(G)$ is a way of gauging connectivity---by measuring how far $G$ is from being disconnected---it contains no information on how disconnected $G$ could become if one removes vertices. On the other hand, though $\omega(G)$ and $\Omega(G)$ gauge how disconnected a graph is, they say nothing about how strongly connected a connected graph is.

Other graph vulnerability parameters do the job of indicating not only how connected a connected graph is, but also how disconnected it could become when vertices are removed. We now define a few such parameters, using our designer/enemy illustration to assist the reader in intuiting what the parameters are measuring.

All of the following vulnerability parameters exhibit, in some way, the relation between the gain and the cost of the enemy in his effort to disconnect the graph. The cost of the enemy will always be the number of vertices $|S|$ he removes. The connectivity $\kappa(G)$ of a graph then corresponds to the minimum cost for disconnecting~$G$.

The enemy's gain will first be measured by $\omega(G - S)$, the number of components in the resulting graph.

If the enemy removes $s$ vertices and breaks the graph into $c$ components, $\frac{s}{c}$ conveys the idea of the (average) cost of each component. In this sense, the \emph{toughness} (introduced in 1973 by Chv\'{a}tal~\cite{chvatal73}) of a graph indicates the enemy's minimum average cost for disconnecting a graph:

\begin{definition}[Toughness]
The toughness $\tau(G)$ of $G$ is defined to be
\[
\min\left\{ \frac{|S|}{\omega(G - S)} \;\middle|\; S \text{ is a cutset of } G \right\}.
\]
A graph $G$ is said to be \emph{$t$-tough} if $\tau(G) \ge t$.
\end{definition}

Then the \emph{scattering number} (introduced in 1978 by Jung~\cite{jung78}) of a graph indicates the maximum ``profit'' (i.e.\ gain $-$ cost) he could hope to achieve when disconnecting this graph:

\begin{definition}[Scattering Number]
The scattering number $\mathrm{sc}(G)$ of a graph $G$ is defined to be
\[
\max\left\{ \omega(G - S) - |S| \;\middle|\; S \text{ is a cutset of } G \right\}.
\]
A graph $G$ is said to be \emph{$s$-unscattered} if $\mathrm{sc}(G) \le s$.
\end{definition}

Suppose the enemy has a sort of ``quota'' to meet before which he gets no gain from his actions. If the enemy must create at least $\ell$ components before counting any gain, what will his minimum cost be for reaching (or exceeding) this quota? The $(\ell)$-connectivity answers this question:

\begin{definition}[$(\ell)$-connectivity]
Given an integer $2 \le \ell \le \alpha(G)$, the $(\ell)$-connectivity (introduced in 1987 by Oellermann~\cite{oellermann87}) $\kappa_\ell(G)$ of $G$ is defined to be\footnote{This parameter should not be confused with the ``generalized $k$-connectivity'' of Hager (1985) and Li--Mao (2014), which measures minimum internally disjoint Steiner trees connecting $k$-vertex subsets and uses the same notation $\kappa_k(G)$.}
\[
\min\left\{ |S| \;\middle|\; \omega(G - S) \ge \ell \right\}.
\]
A graph $G$ is said to be \emph{$(k, \ell)$-connected} if $\kappa_\ell(G) \ge k$.
\end{definition}

Note that $\kappa(G) = \kappa_2(G)$.

The three above vulnerability parameters measure disconnectivity (the gain of the enemy) with $\omega(\cdot)$: The following three parameters use $\Omega(\cdot)$ instead. In contrast with $\omega(\cdot)$, $\Omega(\cdot)$ negatively measures disconnectivity in that smaller $\Omega(\cdot)$ values reflect greater disconnectivity. Note that in this case, the incentive of the enemy would be to decrease $\Omega(\cdot)$. In order to measure disconnectivity positively, the enemy could measure his gains with $n - \Omega(\cdot)$: Indeed, a $0$ gain would correspond to leaving the graph as is and $n - 1$ would correspond to creating an empty graph.

We define three vulnerability parameters: Component-order-toughness (introduced here for sake of completeness), integrity (introduced in 1987 by Barefoot, Entringer, and Swart~\cite{barefoot87}), and Component-order-connectivity (introduced in 1999 by Boesch, Gross, and Suffel~\cite{boesch99}).

\begin{definition}[Component Order Toughness]
The component-order-toughness $\tau^c(G)$ of a graph $G$ is defined to be
\[
\min\left\{ \frac{|S|}{n - \Omega(G - S)} \;\middle|\; S \text{ is a cutset of } G \right\}.
\]
A graph $G$ is said to be \emph{$t$-component-order-tough} if $\tau^c(G) \ge t$.
\end{definition}

\begin{definition}[Integrity]
The integrity $I(G)$ of a graph $G$ is defined to be
\[
\min\left\{ |S| + \Omega(G - S) \;\middle|\; S \subseteq V(G) \right\}.
\]
A graph $G$ is said to be \emph{$i$-integral} if $I(G) \ge i$.
\end{definition}

\begin{definition}[Component Order Connectivity]
The $\ell$-component-order-connectivity of $G$ is defined by
\[
\kappa^c_\ell(G) = \min\left\{ |S| \;\middle|\; \Omega(G - S) < \ell \right\}.
\]
A graph $G$ is said to be \emph{$(k, \ell)$-component-order-connected} (abbreviated $(k, \ell)$-COC) if $\kappa^c_\ell(G) \ge k$.
\end{definition}

Note that $\tau^c(G)$, $n - I(G)$, and $\kappa^c_{n-\ell}(G)$ correspond to $\tau(G)$, $\mathrm{sc}(G)$, and $\kappa_\ell(G)$ respectively, only with the enemy's gain measured by $n - \Omega$ instead of $\omega$.

\section{The Property Function}

In order to extend and unify several vulnerability properties, we make use of a ``property function'' $\psi_G : \N \to \N$ that relates the graph parameters involved in these properties to each other. More precisely, $\psi_G(x)$ will be the minimum cost to achieve a gain of~$x$.

We examine the case where the enemy measures his gain with $\omega(\cdot)$ (his cost still being the number of vertices he must remove). In this case:

\begin{definition}\label{def:propfunc}
Let $\psi_G(x)$ be the minimum number of vertices one must remove from $G$ to be left with exactly $x$ components.
\end{definition}

Formally, $\psi_G(x) = \min\{|S| \mid \omega(G - S) = x\}$. See Figure~\ref{fig:psi} for an example of such a function. In order to intuit this function, we now show and comment on several of its properties:

\begin{fact}\label{fact:psi}
For all graphs $G$ of order $n$,
\begin{enumerate}
\item[(a)] $\psi_G$ is well defined (only) for all $0 \le x \le \alpha(G)$.
\item[(b)] $\psi_G(x) \le n - x$.
\item[(c)] For all $x$ in the domain of $\psi_G$ and all $y$ with $\psi_G(x) \le y \le n - x$, there exists an $S \subseteq V(G)$ such that $|S| = y$ and $\omega(G - S) = x$.
\item[(d)] $\omega(G) = \psi_G^{-1}(0)$.
\item[(e)] $\psi_G(\omega(G) - i) - \psi_G(\omega(G) - i + 1)$ is the order of the $i$-th component of $G$ (in increasing order).
\end{enumerate}
\end{fact}

\begin{proof}
The independence number $\alpha(G)$ is the maximum possible order of a subset of independent vertices of~$G$. Note that $\alpha(G)$ can be thought of as a vulnerability parameter in that it indicates the maximum number of isolated vertices the enemy could create by removing vertices of~$G$. Since a component has at least one vertex, no $S \subset V(G)$ could possibly yield $\omega(G - S) > \alpha(G)$, so $\psi_G(x)$ is undefined for $x > \alpha(G)$.

Now $\psi_G(\alpha(G))$ is defined since $G - I$, where $I$ is a maximal independent set of vertices, is a cutset creating exactly $\alpha(G)$ components, and removing these components (composed of single vertices) one by one, we can create any number $x$ of components, for $x$ ranging from $\alpha(G)$ to $0$. All graphs verify $\psi_G(0) = n$, since one would have to remove all vertices of the graph to be left with no components.

Since a component has at least one vertex, no $S \subset V(G)$ could possibly yield $\omega(G - S) > \alpha(G)$.

Part~(b) is true since a component must have at least 1 vertex.

Part~(c) comes from the fact that a connected component always contains a vertex that does not disconnect it and shows that $\psi_G$ gives us full information not only on the optimal, but on all possible (cost, gain) pairs the enemy can achieve.

Though there could be several values of $x$ for which $\psi_G(x) = y$ if $y > 0$, there is obviously only one $x$ for which $\psi_G(x) = 0$: Namely $x = \omega(G)$. This is what (d) states.

This implies that $G$ is connected if, and only if, $\psi_G(1) = 0$, and in this case (e) translates to $\psi_G(0) - \psi_G(1) = n$, indeed the order of the only component of $G$: Itself. If $G$ has $z$ components, in order to create $z - 1$ components, the enemy must totally remove one of the components, thus at least as many vertices as the smallest component has. Hence $\psi_G(\omega(G) - 1) = \psi_G(\omega(G) - 1) - \psi_G(\omega(G))$ is the order of the smallest component, $\psi_G(\omega(G) - 2) - \psi_G(\omega(G) - 1)$ is the order of the second smallest, etc.
\end{proof}

\begin{figure}[ht]
\centering
\includegraphics[width=0.7\textwidth]{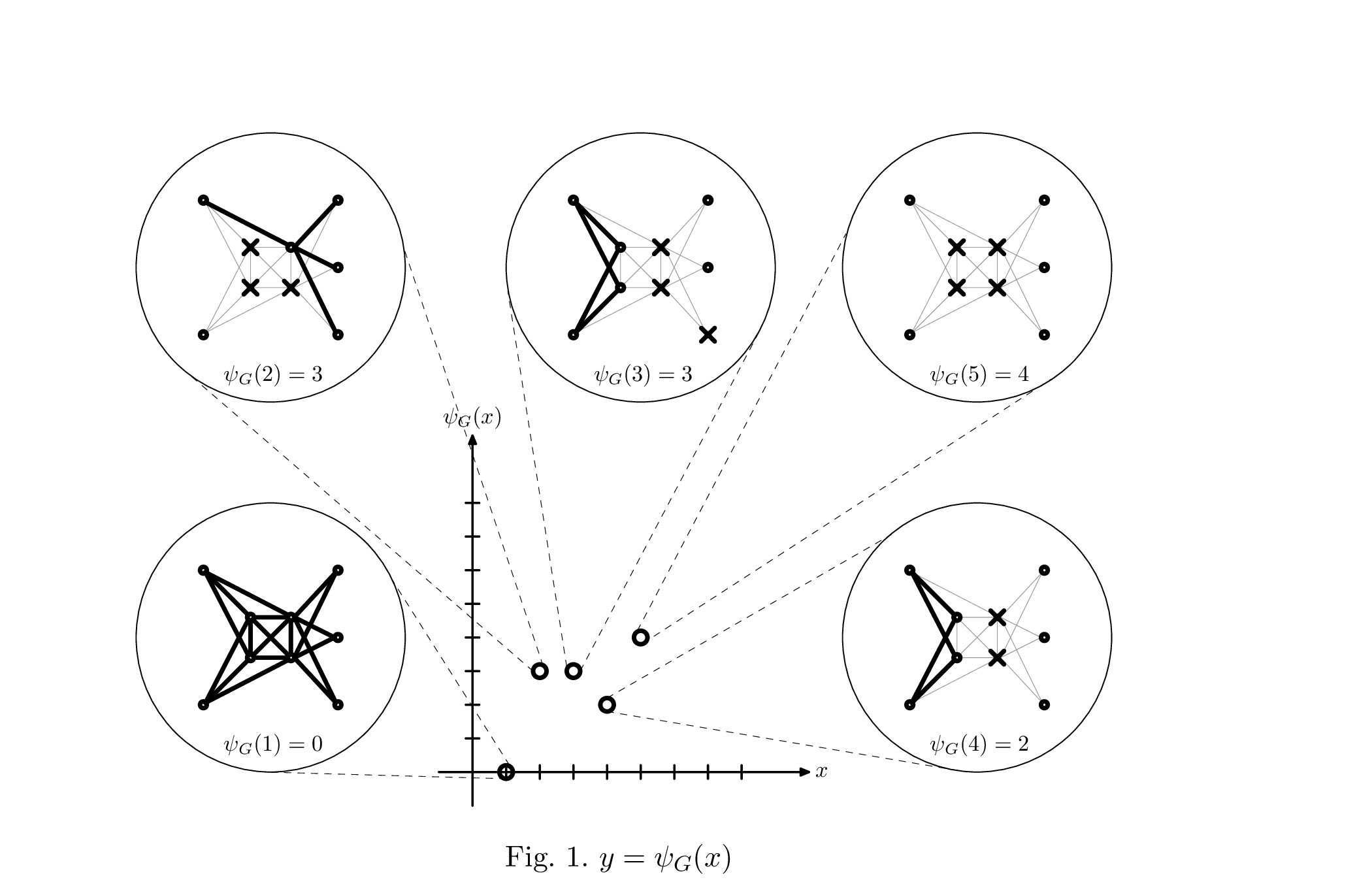}
\caption{$y = \psi_G(x)$}
\label{fig:psi}
\end{figure}

Moreover, we can extract from $\psi_G$ several of the vulnerability parameters we've introduced earlier:
\begin{enumerate}
\item[(1)] $\kappa(G) = \min\{\psi_G(x) \mid x \ge 2\}$
\item[(2)] $\tau(G) = \min\left\{\dfrac{\psi_G(x)}{x} \;\middle|\; x \ge 2\right\}$
\item[(3)] $\mathrm{sc}(G) = \max\{x - \psi_G(x) \mid x \ge 2\}$
\item[(4)] $\kappa_\ell(G) = \min\{\psi_G(x) \mid x \ge \ell\}$
\end{enumerate}

Therefore the properties of connectivity, toughness, ``unscatteredness'' and $(\ell)$-connectivity all convey the fact that the property function is above a given line, on the given interval $[\ell, \alpha]$. We are naturally led to define the following general vulnerability property.

\begin{definition}[$(f, \ell)$-connectivity]
Let $f : \R \to \R$ and $\ell \ge 0$. A graph $G$ is said to be \emph{$(f, \ell)$-connected} if $\psi_G(x) \ge f(x)$ for all $\ell \le x \le \alpha(G)$. If $f$ is the linear function $f(x) = tx + k$ we say that $G$ is \emph{$(t, k, \ell)$-connected}.
\end{definition}

This gives us a parametrization of vulnerability that is finer than, yet consistent with, the vulnerability parameters $\kappa$, $\kappa_\ell$, $\tau$, and $\mathrm{sc}$. Indeed, if $t \ge t'$, $k \ge k'$, and $\ell \le \ell'$, then $(t, k, \ell)$-connected implies $(t', k', \ell')$-connected, and we have the following correspondences:

\begin{fact}\label{fact:correspondences}
\begin{align*}
G \text{ is $k$-connected} &\iff G \text{ is $(0, k, 2)$-connected} \\
G \text{ is $t$-tough} &\iff G \text{ is $(t, 0, 2)$-connected} \\
G \text{ is $s$-unscattered} &\iff G \text{ is $(1, -s, 2)$-connected} \\
G \text{ is $(k, \ell)$-connected} &\iff G \text{ is $(0, k, \ell)$-connected}
\end{align*}
\end{fact}

Saying that $\psi_G(x) \ge f(x)$ for all $\ell \le x \le \alpha(G)$ is equivalent to saying that the function $\psi_G$ does not intersect the region $R$ bounded by $x = \ell$, $y = 0$, $y = n - x$, and $y = f(x)$---where (only) the last boundary is not included in the region. If $P$ is a graph property we write $G \in P$ to indicate that $G$ has property~$P$, and $P \subseteq P'$ to mean that whenever a graph has property~$P$, it necessarily has property~$P'$. In general, it will often be useful to express the fact that $G \in P$ as $\psi_G \cap R_P = \emptyset$ where $R_P$ is the forbidden region associated to~$P$, and $\psi_G$ is identified with the set of pairs $(x, \psi_G(x))$.

Note that $P \subseteq P'$ if, and only if, $R_P \subseteq R_{P'}$. This perspective allows us to relate the different vulnerability parameters to each other with ease, establishing several inequalities that can be found throughout the literature. For example, the inequalities of Fact~\ref{fact:inequalities} can readily be verified graphically (see Table~\ref{tab:implications}).

\begin{fact}\label{fact:inequalities}
Given a graph $G$ of order $n$,
\begin{enumerate}
\item[(1)] $\kappa(G) \ge 2\tau(G)$ as observed in~\cite{chvatal73}.
\item[(2)] $\tau(G) \ge \frac{\kappa(G)}{\alpha(G)}$ as observed in~\cite{chvatal73}.
\item[(3)] $\tau(G) \ge \frac{k}{n - k}$ for $k$-connected graphs.
\item[(4)] $\alpha(G) \le n - \kappa(G)$.
\item[(5)] $\alpha(G) \le \frac{n}{\tau(G) + 1}$ as observed in~\cite{chvatal73} and~\cite{alon95}.
\item[(6)] $\mathrm{sc}(G) \le \alpha(G) - \kappa(G)$ as observed in~\cite{kirlangic02}.
\item[(7)] $\mathrm{sc}(G) \le \max\{1, n - 2k\}$ as observed in~\cite{zhang01}.
\item[(8)] $\mathrm{sc}(G) \le \max\left\{1,\; n\,\frac{1 - \tau(G)}{\tau(G) + 1}\right\}$.
\end{enumerate}
\end{fact}

\begin{table}[ht]
\centering
\caption{Implications. In each of these cases, $R$ and $R'$ are regions which, when forbidden, correspond to graph properties. In each case, $R' \subseteq R$ so if $R$ is forbidden, $R'$ is necessarily forbidden. Therefore the property corresponding to $R$ implies the property corresponding to $R'$.}
\label{tab:implications}
\vspace{1em}
\includegraphics[width=0.9\textwidth]{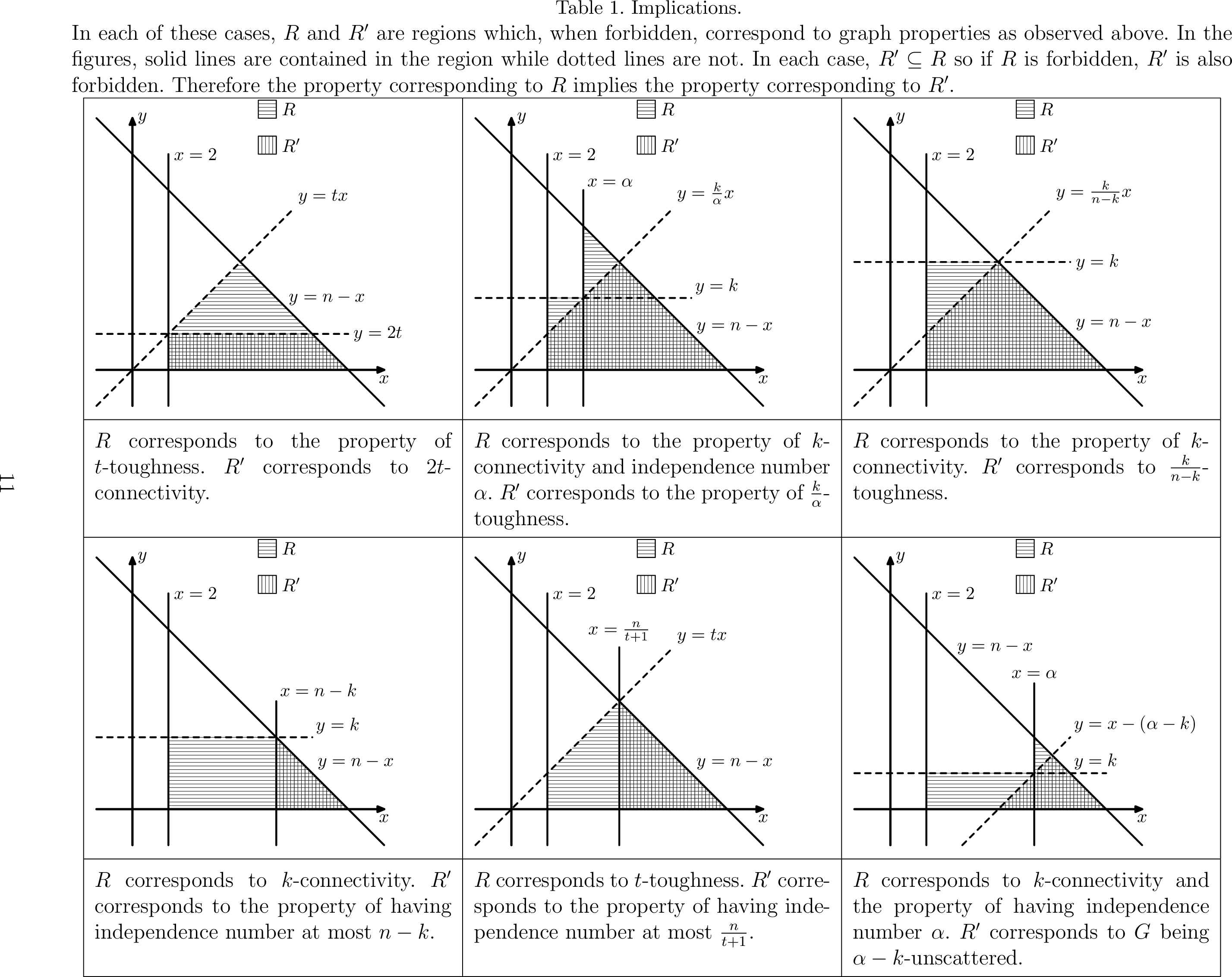}
\end{table}

\section{The Density Function}

Given two graphs $H$ and $F$, we write $H \preceq F$ if $|H| = |F|$ and $H$ is a subgraph of~$F$. A graph parameter $\mu : \G \to \R$ is said to be \emph{increasing} (resp.\ \emph{decreasing}) if adding edges to a graph will not decrease (resp.\ increase) the value of the parameter. A \emph{monotone} parameter is one that is either increasing or decreasing.

Formally, $\mu$ is increasing (resp.\ decreasing) if for all graphs, $H \preceq F \Rightarrow \mu(H) \le \mu(F)$ (resp.\ $F \preceq H \Rightarrow \mu(H) \le \mu(F)$).

Note that $\Omega$, $\kappa$, $\tau$, $\kappa_\ell$, $\tau^c$, $I$, and $\kappa^c_\ell$ are increasing, whereas $\alpha$, $\omega$, and $\mathrm{sc}$ are decreasing.

A graph property $P$ is \emph{increasing} if for any $H \in P$, $H \preceq F \Rightarrow F \in P$.

Any increasing (resp.\ decreasing) parameter $\mu : \G \to \R$ can naturally be associated with a family $P^\mu_m$ of increasing ``$\mu$-properties'' where, for $r \in \R$,
\[
P^\mu_m = \{G \mid \mu(G) \ge m\} \quad (\text{resp.}\ P^\mu_m = \{G \mid \mu(G) \le m\}).
\]

Technically, many of the parameters we've introduced so far are not monotone in that they are not defined for all graphs; namely on the more (edge-)dense graphs. More precisely, $\kappa$, $\tau$, $\mathrm{sc}$, and $\tau^c$ were not defined for the complete graph $K_n$, and $\kappa_\ell$ was not defined for any graphs $G$ with $\alpha(G) < \ell$. Yet these can all naturally be extended to be monotone parameters by letting $\kappa(K_n) = n - 1$, $\tau(K_n) = n - 1$, $\mathrm{sc}(K_n) = -n$, $\tau^c(K_n) = 1$, and for all $G$ such that $\alpha(G) < \ell$, $\kappa_\ell(G) = n - \alpha(G)$.

So monotone parameters $\mu$ all reflect, in their own way, how dense a graph is, but their associated increasing properties $P^\mu_m$ all converge towards the densest graph of all---the complete graph $K_n$. Therefore, for any fixed $n$, any increasing property $P$, and any monotone parameter $\mu$, there exists an $m$ such that $P^\mu_m \subseteq P$ (since there are only finitely many graphs of order~$n$).

In the case of an increasing parameter this means that we can always choose $m$ large enough so that $\mu(G) \ge m$ implies that $G$ has property~$P$. But how large must $m$ be to oblige $P$? In other words, how large can $\mu(G)$ be if $G \notin P$?

From now on, ``optimal'' should be understood to mean maximal when we are referring to increasing monotonicity and minimal when we are referring to decreasing monotonicity. We will also use the notation $\mathrm{opt}$ instead of $\max$ and $\min$ in order to make general statements.

\begin{definition}[$\mu$-threshold]
Given a monotone parameter $\mu$ and monotone property $P$, the \emph{$\mu$-threshold} for property $P$ is defined to be the optimal value $T$ of $\mu(G)$ over all graphs $G$ which do not have property~$P$.
\end{definition}

In the case of an increasing parameter $\mu$, the number $T$ is a threshold for $P$ in that there are graphs $G$ with $\mu(G) \le T$ that don't have the property but as soon as $\mu(G) > T$, $G$ must necessarily have the property.

We have seen how the vulnerability properties we introduced earlier could be encoded by an appropriate ``property function''~$\psi$. We now present a method to determine the $\mu$-threshold of these properties by transforming the problem to the task of optimizing a given function (determined by~$\mu$) over a given region (determined by the particular property we are looking at) of the range of the property function. For this purpose, we define the ``density function'' $\varphi^\psi_\mu$ for the parameter~$\mu$ and property function~$\psi$:

\begin{definition}
Given a density parameter $\mu$, a ``property function''~$\psi$, and number of vertices~$n$, let
\[
\varphi^\psi_\mu(x, y) := \mathrm{opt}\{\mu(G) \mid |G| = n,\; \psi_G(x) = y\}.
\]
\end{definition}

In other words, $\varphi^\psi_\mu(x, y)$ is the maximum value $\mu(G)$ can attain for a graph $G$ of order $n$ satisfying $\psi_G(x) = y$. When the context is clear, we will drop the superscript~$\psi$. See Figure~\ref{fig:phi} for two examples of $\varphi$ functions, one using the edge count $\mu(G) = |E(G)|$ and the other using the minimum degree $\mu(G) = \delta(G)$.

\begin{figure}[ht]
\centering
\includegraphics[width=0.9\textwidth]{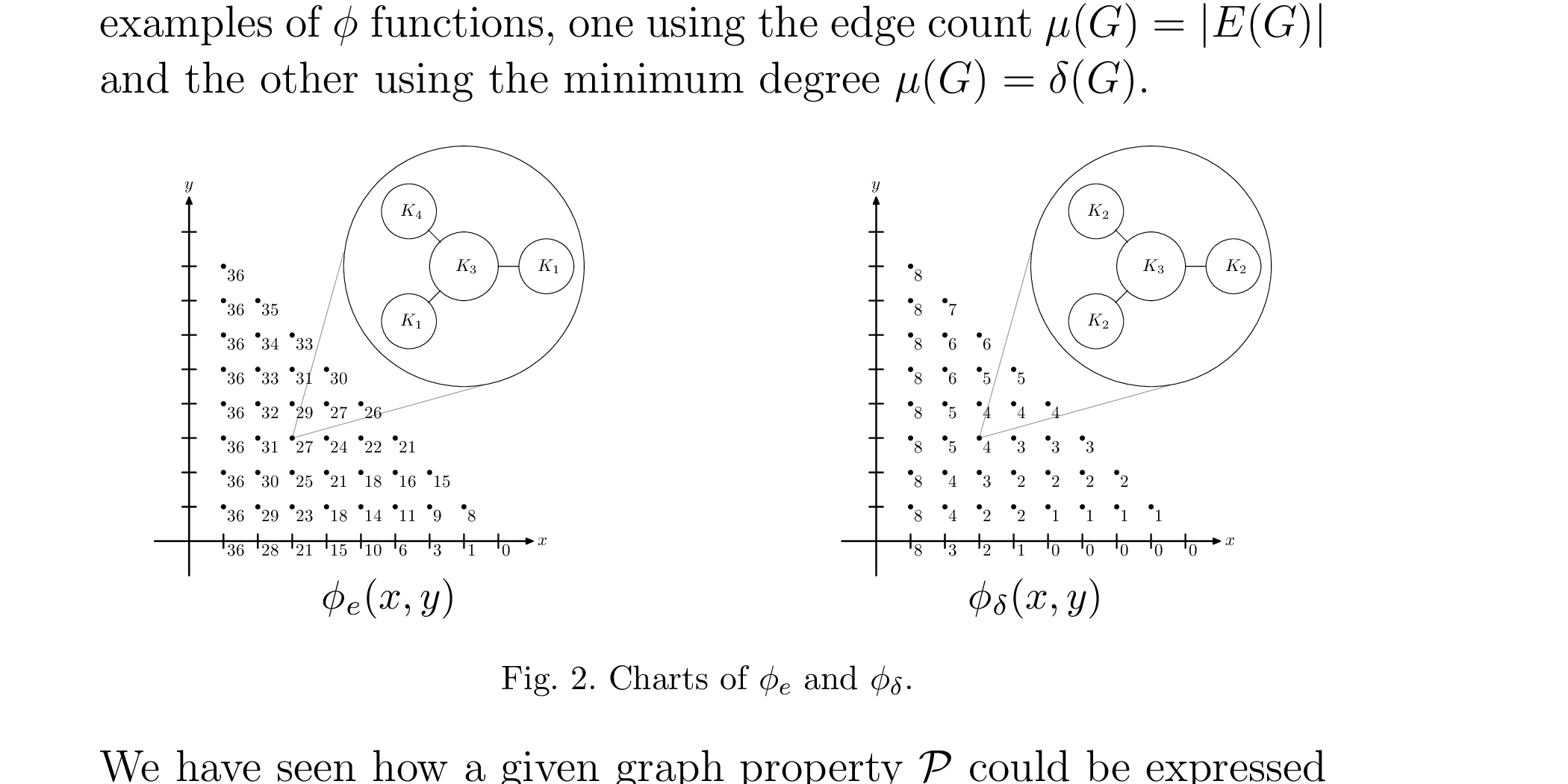}
\caption{Charts of $\varphi_e$ and $\varphi_\delta$.}
\label{fig:phi}
\end{figure}

We have seen how a given graph property $P$ could be expressed by forbidding $\psi$ from intersecting a given region $R_P$ of $\N^2$ (see Table~\ref{tab:implications}). This perspective facilitates finding the $\mu$-threshold: Indeed, the $\mu$-threshold for $P$ will simply be the maximum value of $\varphi^\psi_\mu(x, y)$ over the region $R_P$, as stated in the following Lemma.

\begin{lemma}\label{lem:threshold}
The $\mu$-threshold for $P$ is $T = \mathrm{opt}_{R_P} \varphi^\psi_\mu$.
\end{lemma}

\begin{proof}
We will prove the case where $\mu$ is monotone increasing, the decreasing case being similar.

By the definition of $\varphi^\psi_\mu$, for any graph $G$ and $(x, y) \in \psi_G$, $\mu(G)$ can be no more than $\varphi^\psi_\mu(x, y)$. Therefore,
\begin{align*}
\mu(G) &\le \min\{\mu(G) \mid (x, y) \in \psi_G\} \\ 
       &\le \min\{\mu(G) \mid (x, y) \in \psi_G \cap R_P\} \\
       &\le \max\{\mu(G) \mid (x, y) \in \psi_G \cap R_P\} \\
       &\le \max\{\mu(G) \mid (x, y) \in R_P\}.
\end{align*}
Thus, $T \le \max_{R_P} \varphi^\psi_\mu$. Now let $(x_0, y_0) \in R_P$ be a pair that maximizes $\varphi^\psi_\mu$ in $R_P$ and $G$ be a graph that maximizes $\mu$ in $\{G \in \G \mid \psi_G(x_0) = y_0\}$. Observe that $G \notin P$, so $T \ge \mu(G) = \max_{R_P} \varphi^\psi_\mu$, proving the Lemma.
\end{proof}

\section{Results}

In the sections above we have developed tools to determine, for any two monotone parameters $\mu$ and $\rho$, the $\mu$-threshold of a $\rho$-property. We will exemplify this approach on the vulnerability parameters introduced earlier. Fact~\ref{fact:psi} hinted at some of these thresholds, but tightness was not proven. Also, we will add to our list three customary density parameters: The edge count $e(G)$, the minimum degree $\delta(G)$, and the neighborhood intersection $N^j_\cap(G)$---the minimum, over all choices of $j$ vertices from $V(G)$, order of the common neighborhood of those $j$ vertices. This last parameter has been studied as a density measure in sufficient conditions for hamiltonicity~\cite{gould03,fan04}, as a property of random graphs~\cite{gerke04} and it has arisen in a variety of other situations as well. Note that $N^j_\cap(G)$ extends the minimum degree in that $N^1_\cap(G) = \delta(G)$.

To summarize, in order to find the $\mu$-threshold of a property $P$, one can proceed in four phases as follows:

\medskip
\noindent\textbf{I.}~Define a property function $\psi$ such that for some region $R_P$, we have $G \in P \iff \psi_G \cap R_P = \emptyset$.

\noindent\textbf{II.}~Compute the density function $\varphi^\psi_\mu$.

\noindent\textbf{III.}~Optimize $\varphi^\psi_\mu$ over $R_P$ to get the desired threshold.

\medskip
Regarding phase~\textbf{I}, for the vulnerability properties associated to parameters involving $\omega$ (namely, $\kappa$, $\tau$, $\mathrm{sc}$, and $\kappa_\ell$) we will use the property function $\psi^\omega_G(x) = \min\{|S| \mid \omega(G - S) = x\}$. For vulnerability properties involving $\Omega$ (namely, $\tau^c$, $I$, and $\kappa^c_\ell$) we will use the property function $\psi^\Omega_G(x) = \min\{|S| \mid \Omega(G - S) = x\}$.

The following considerations will be useful for phase~\textbf{II} of our approach. Given $y \in \N$ and a vector of positive integers $\vec{x} = (x_1, \ldots, x_z)$, let $\G_n(\vec{x}, y)$ be the family of graphs of order $n$ that can be decomposed into $z + 1$ disjoint subgraphs $Y, X_1, \ldots, X_z$ where $|Y| = y$, $(|X_1|, \ldots, |X_z|) = \vec{x}$, and $X_1, \ldots, X_z$ are connected mutually independent subgraphs (i.e.\ $E(X_i, X_j) = \emptyset$ for all $i \ne j$). Further, we define
\[
K(\vec{x}, y) = K_y + \left(\bigcup_{i=1}^{z} K_{x_i}\right),
\]
where $+$ represents the join operation (whereby all possible edges between the operand graphs are inserted).

Observe that for any density parameter $\mu$, the graph $K(\vec{x}, y)$ optimizes $\mu$ within $\G_n(\vec{x}, y)$. This implies that to compute $\varphi^\psi_\mu(x, y)$ we need only to optimize $\mu$ over those $K(\vec{x}, y)$ that satisfy $\psi_{K(\vec{x}, y)}(x) = y$.

Let $|\vec{x}| = z$, $\|\vec{x}\|_\infty = \max_i x_i$, and $\|\vec{x}\|_0 = \min_i x_i$. The reader may verify that $\psi^\omega_G(x) = y$ if, and only if, $G \in \bigcup_{|\vec{x}| = x} \G_n(\vec{x}, y)$, so
\begin{equation}\label{eq:phi_omega}
\varphi^{\psi^\omega}_\mu(x, y) = \mathrm{opt}\{\mu(K(\vec{x}, y)) \mid |\vec{x}| = x\}.
\end{equation}
Similarly, $\psi^\Omega_G(x) = y$ if, and only if, $G \in \bigcup_{\|\vec{x}\|_\infty = x} \G_n(\vec{x}, y)$, so
\begin{equation}\label{eq:phi_Omega}
\varphi^{\psi^\Omega}_\mu(x, y) = \mathrm{opt}\{\mu(K(\vec{x}, y)) \mid \|\vec{x}\|_\infty = x\}.
\end{equation}

Hence we first compute $\mu(K(\vec{x}, y))$ for our parameters of interest:

\begin{center}
\begin{tabular}{ll}
\toprule
$\mu$ & $\mu(K(\vec{x}, y))$ \\
\midrule
$\delta$ & $y + \|\vec{x}\|_0$ \\
$e$ & $y(n - y) + \binom{y}{2} + \sum \binom{x_i}{2}$ \\
$N^j_\cap$ & $y$ \quad (for $j \ge 2$) \\
$\kappa$ & $y$ \\
$\kappa_\ell$ & $\begin{cases} y & \text{if } \ell \le |\vec{x}| \\ n - \ell & \text{if } \ell > |\vec{x}| \end{cases}$ \\
$\tau$ & $\dfrac{y}{|\vec{x}|}$ \\
$\mathrm{sc}$ & $|\vec{x}| - y$ \\
$\kappa^c_\ell$ & $\begin{cases} y & \text{if } \ell \ge \|\vec{x}\|_\infty \\ n - \|\vec{x}\|_\infty & \text{if } \ell < \|\vec{x}\|_\infty \end{cases}$ \\
$I$ & $\|\vec{x}\|_\infty + y$ \\
$\tau^c$ & $\dfrac{y(\|\vec{x}\|_\infty + n)}{n^2}$ \\
\bottomrule
\end{tabular}
\end{center}

Using~\eqref{eq:phi_omega} and~\eqref{eq:phi_Omega}, we can now compute our density functions (see Table~\ref{tab:density}).

Finally, we can determine the $\mu$-thresholds of our properties of interest by optimizing the density function over the appropriate region. Results are exhibited below (see Tables~\ref{tab:delta_results} and~\ref{tab:property_thresholds}).

\begin{table}[p]
\centering
\caption{Table of density functions. Here $\gamma = \frac{n-y}{x}$, $\lambda = \frac{n - y - x}{x}$, $r$ is the remainder of the Euclidean division of $n - y$ by $x$, and $r'$ is the remainder of the Euclidean division of $n - y - x$ by $\lceil\lambda\rceil$.}
\label{tab:density}
\vspace{1em}

\renewcommand{\arraystretch}{1.8}
\begin{tabular}{l|ccc}
\toprule
$\mu$ & $\delta$ & $e$ & $N^j_\cap$ \\
\midrule
$\varphi^{\psi^\omega}_\mu$ & $\left\lfloor\gamma\right\rfloor + y - 1$ & $\binom{n-x+1}{2} + y(x-1)$ & $\begin{cases} y & j \le n - y \\ n - j & j > n - y \end{cases}$ \\[6pt]
$\varphi^{\psi^\Omega}_\mu$ & $\left\lfloor\frac{n-y-x}{\lceil\gamma\rceil - 1}\right\rfloor + y - 1$ & $\frac{y^2 + ny - y}{2} + \lfloor\gamma\rfloor\binom{x}{2} + \binom{r}{2}$ & $\begin{cases} y & j \le n - y \\ n - j & j > n - y \end{cases}$ \\
\bottomrule
\end{tabular}

\vspace{2em}

\renewcommand{\arraystretch}{1.8}
\begin{tabular}{l|cccc}
\toprule
$\mu$ & $\kappa$ & $\tau$ & $\mathrm{sc}$ & $\ell$-conn. \\
\midrule
$\varphi^{\psi^\omega}_\mu$ & $y$ & $\dfrac{y}{x}$ & $x - y$ & $\begin{cases} y & x \ge \ell \\ n - \ell & x < \ell \end{cases}$ \\[8pt]
$\varphi^{\psi^\Omega}_\mu$ & $y$ & $\dfrac{y}{\lceil\lambda\rceil + 1}$ & $\lceil\lambda\rceil + 1 - y$ & $\begin{cases} y & \lceil\lambda\rceil \ge \ell \\ n - \ell & \lceil\lambda\rceil < \ell \end{cases}$ \\
\bottomrule
\end{tabular}

\vspace{2em}

\renewcommand{\arraystretch}{1.8}
\begin{tabular}{l|ccc}
\toprule
$\mu$ & $\tau^c$ & int. & $\kappa^c_\ell$ \\
\midrule
$\varphi^{\psi^\omega}_\mu$ & $\dfrac{y}{y + x - 1}$ & $n - x + 1$ & $\begin{cases} y & \ell > n - y - x + 1 \\ n - x - \ell & \ell \le n - y - x + 1 \end{cases}$ \\[8pt]
$\varphi^{\psi^\Omega}_\mu$ & $\dfrac{y}{n - x}$ & $x + y$ & (see text) \\
\bottomrule
\end{tabular}
\end{table}

\begin{table}[p]
\centering
\caption{Table of results for $\delta(G)$ and $e(G)$.}
\label{tab:delta_results}
\vspace{1em}

\renewcommand{\arraystretch}{2.2}
\begin{tabular}{l|l}
\toprule
Property & $\delta$-Threshold \\
\midrule
$k$-Connectivity & $\left\lfloor\dfrac{n + k - 3}{2}\right\rfloor$ \quad \cite{chartrand68,boesch74} \\[6pt]
$t$-Toughness ($t > 0$) & $\max\left\{\left\lfloor\dfrac{n + 2t - 3}{2}\right\rfloor,\; \left\lfloor\dfrac{n - \lceil t\lfloor\gamma\rfloor\rceil}{\lfloor\gamma\rfloor}\right\rfloor + \lceil t\lfloor\gamma\rfloor\rceil - 1,\; \left\lfloor\dfrac{\lceil\gamma\rceil + 1}{\lceil\gamma\rceil}\right\rfloor + n - \lceil\gamma\rceil - 2\right\}$ \\
& where $\gamma = \frac{n}{t+1}$. This improves a bound in~\cite{bauer90}. If $t \le 0$, the threshold is $\left\lfloor\frac{n + 2t - 3}{2}\right\rfloor$. \\[6pt]
$s$-Unscattered & $\max\left\{\left\lfloor\dfrac{n - s - 1}{2}\right\rfloor,\; \left\lfloor\dfrac{\lceil n/2\rceil + s}{\lfloor n/2 \rfloor}\right\rfloor + \left\lceil\dfrac{n}{2}\right\rceil + s - 1,\; \left\lceil\dfrac{n}{2}\right\rceil - 1\right\}$ \\[6pt]
$\kappa_\ell \ge k$ & $\left\lfloor\dfrac{n + k\ell}{k}\right\rfloor + k - 1$ \quad \cite{oellermann87} \\[6pt]
$i$-Integrity & $i - 2$ \quad \cite{goddard90} \\[6pt]
$(k, \ell)$-$\kappa^c$ & $\left\lfloor\dfrac{n - k - \xi + 1}{\left\lceil\frac{n - k + 1}{\xi}\right\rceil - 1}\right\rfloor + k - 2$ \\
& where $\xi = \begin{cases} \ell & \text{if } \ell \text{ divides } n - k - 1 \\ \ell - \left(\left\lceil\frac{n-k-1}{\ell}\right\rceil\cdot\ell - (n-k-1)\right) & \text{otherwise} \end{cases}$ \quad [No prior bounds existed for this parameter, which was only recently defined in~\cite{boesch99}.] \\
\bottomrule
\end{tabular}
\end{table}

\begin{table}[p]
\centering
\caption{Property thresholds. The thresholds preceded by $\sim$ are approximate (see text for exact expressions).}
\label{tab:property_thresholds}
\vspace{1em}

\renewcommand{\arraystretch}{2.0}
\resizebox{\textwidth}{!}{%
\begin{tabular}{l|cccccc}
\toprule
 & $\kappa(G) \ge k$ & $\kappa_\ell \ge k$ & $\tau(G) \ge t$ & $\mathrm{sc}(G) \le s$ & $\kappa^c_\ell \ge k$ & $I(G) \ge i$ \\
\midrule
$\kappa(G)$ & $k - 1$ & $k - 1$ & $\left\lceil\dfrac{nt}{t+1}\right\rceil - 1$ & $\left\lfloor\dfrac{n-s}{2}\right\rfloor - 1$ & $k - 1$ & $i - 2$ \\[6pt]
$\tau(G)$ & $\dfrac{k-1}{2}$ & $\dfrac{k-1}{\ell}$ & $\sim t$ & * & $\sim \dfrac{(\ell-1)(k-1)}{n-k+1}$ & $\sim \dfrac{i-2}{n-i+2}$ \\[6pt]
$\mathrm{sc}(G)$ & $3 - k$ & $\ell - k + 1$ & ** & $s + 1$ & $\sim \dfrac{n-k+1}{\ell-1} - k + 2$ & $\sim \dfrac{n-1}{i-2}$ \\[6pt]
$I(G)$ & $n - 1$ & $n - \ell + 1$ & $n - 1$ & $n - s - 1$ & $k + \ell - 1$ & $i - 1$ \\[6pt]
$\tau^c(G)$ & $\dfrac{k-1}{k}$ & $\dfrac{k-1}{k + \ell - 2}$ & $\sim \dfrac{t}{t+1}$ & *** & $\dfrac{k}{n - \ell + 1}$ & $\dfrac{1}{n - i + 1}$ \\
\bottomrule
\end{tabular}%
}

\vspace{1em}
{\small
* $\frac{n-s-2}{n+s-2}$ if $s > 0$;\quad $\frac{\lceil n/2\rceil - 1}{\lfloor n/2\rfloor + 2}$ if $s = 0$;\quad $\frac{1-s}{2}$ if $s < 0$.

** $\frac{1}{t}$ if $t \le 1$;\quad $2\left\lfloor\frac{n}{t+1}\right\rfloor + 2 - n$ if $t > 1$.

*** $\frac{\lfloor(x+s)/2\rfloor - s - 1}{2\lfloor(x-s)/2\rfloor - s - 2}$
}
\end{table}

\clearpage

As we have seen, $k$-connectivity, $t$-toughness, $s$-unscatteredness, and $(k, \ell)$-connectivity are all instances of the more general $(t, k, \ell)$-connectivity we've introduced. To conclude, we will hence determine the expression of the $\delta$-threshold for $(t, k, \ell)$-connectivity. Other $\mu$-thresholds may be determined similarly.

\begin{theorem}\label{thm:main}
Given a linear function $f(x) = tx + k$, the $\delta$-threshold for $f$-connectivity is given by:
\[
\max\left\{
\left\lfloor\frac{n + (\ell - 1)\lceil\ell t + k - 1\rceil}{\ell}\right\rfloor - 1,\quad
\left\lfloor\frac{n - \lceil t\lfloor\gamma\rfloor\rceil - k}{\lfloor\gamma\rfloor}\right\rfloor + \lceil t\lfloor\gamma\rfloor\rceil - k - 1,\quad
\left\lfloor\frac{\lceil\gamma\rceil + 1}{\lceil\gamma\rceil}\right\rfloor + n - \lceil\gamma\rceil - 2
\right\}
\]
where $\gamma = \frac{n - k}{t + 1}$ as long as $t > 0$. If $t \le 0$, the threshold is
$\left\lfloor\frac{n + (\ell - 1)\lceil\ell t + k - 1\rceil}{\ell}\right\rfloor - 1$.
\end{theorem}

\begin{proof}
We use the property function $\psi = \psi^\omega$ defined above. The $\delta$-threshold $T$ is the maximum of $\varphi^\psi_\delta(x, y) = \left\lfloor\frac{n - y}{x}\right\rfloor + y - 1$ (see Table~\ref{tab:density}) over
\[
R = \{(x, y) \mid 0 \le y \le \max(\lceil f(x)\rceil - 1,\; n - x),\; \ell \le x \le \alpha(G)\}.
\]

Notice that for $x \ge 1$,
\[
\varphi(x, y + 1) = \left\lfloor\frac{n + (x-1)(y+1)}{x}\right\rfloor - 1
\ge \left\lfloor\frac{n + (x-1)y}{x}\right\rfloor - 1
= \varphi(x, y).
\]
Therefore $\varphi(x, y)$ is nondecreasing with $y$, so letting $A = [\ell, \alpha(G)]$, $A_1 = \left[\ell, \left\lceil\frac{n-k}{t+1}\right\rceil\right]$, and $A_2 = \left[\left\lceil\frac{n-k}{t+1}\right\rceil, \alpha(G)\right]$, we get
\[
T = \max_{x \in A} \varphi\!\left(x,\; \max\{\lceil f(x)\rceil - 1,\; n - x\}\right)
= \max\!\left\{\max_{x \in A_1} \varphi\!\left(x,\; \lceil f(x)\rceil - 1\right),\; \max_{x \in A_2} \varphi\!\left(x,\; n - x\right)\right\}.
\]

Now,
\[
\max_{x \in A_2} \varphi(x, n - x) = \varphi\!\left(\left\lceil\frac{n-k}{t+1}\right\rceil,\; n - \left\lceil\frac{n-k}{t+1}\right\rceil\right)
\]
because $\varphi$ is non-increasing as $x$ increases, so we are left with finding $\max_{x \in A_1} \varphi(x, \lceil f(x)\rceil - 1)$.

\[
\max_{x \in A_1} \varphi(x, \lceil f(x)\rceil - 1) = \max_{x \in A_1} \left\{\left\lfloor\frac{n - \lceil tx\rceil - k}{x}\right\rfloor + \lceil tx\rceil + k\right\}.
\]
If $t > 0$, one may show that the maximum value of $\varphi$ occurs at $(\ell,\; \ell t + k - 1)$ or $\left(\left\lfloor\frac{n-k}{t+1}\right\rfloor,\; \left\lceil f\!\left(\left\lfloor\frac{n-k}{t+1}\right\rfloor\right)\right\rceil - 1\right)$. If $t \le 0$, the maximum value of $\varphi$ occurs at $(\ell,\; \ell t + k - 1)$.
\end{proof}

\subsection*{Note added in 2026}

Since this paper was originally drafted, several developments in the literature have paralleled or complemented our framework:

\begin{itemize}
\item Bauer, Broersma, and Schmeichel~\cite{bauer06} published a comprehensive survey of toughness results. Bauer, Broersma, van den Heuvel, Kahl, and others~\cite{bauer15} surveyed best monotone \emph{degree-sequence} conditions (Chv\'{a}tal-type) for connectivity, toughness, binding number, $k$-factors, and Hamiltonicity. Our minimum-degree thresholds are special cases of their more refined degree-sequence conditions, though our unified derivation via $\psi_G$ and $\varphi^\psi_\mu$ remains distinct.

\item Kahl~\cite{kahl19,kahl21} proved that vertex compression cannot increase any of the standard vulnerability parameters, showing that threshold and quasi-threshold graphs are extremal. This provides a structural complement to our algebraic/geometric approach: Kahl identifies \emph{which} graphs are extremal; our framework identifies \emph{what the thresholds are}.

\item Gu~\cite{gu21} proved Brouwer's conjecture~\cite{brouwer95} that $\tau(G) \ge d/\lambda - 1$ for $d$-regular graphs, where $\lambda$ is the second-largest absolute eigenvalue. Gu and Haemers~\cite{gu22} extended this to Laplacian eigenvalues, and Chen, Li, and Shiu~\cite{chen24} established spectral bounds for scattering number, integrity, tenacity, and $\ell$-connectivity simultaneously.

\item On the computational side, Drange, Dregi, and van~'t~Hof~\cite{drange16} showed that component-order connectivity is FPT parameterized by $k + \ell$, and Kumar and Lokshtanov~\cite{kumar16} gave a $2\ell k$-vertex kernel. Katona and Khan~\cite{katona24} proved toughness is FPT by treewidth.

\item The component-order connectivity model has been surveyed comprehensively by Gross et~al.~\cite{gross13} and extended to directed graphs by Bang-Jensen et~al.~\cite{bangjensen22}.
\end{itemize}

\section{Further Research}

The framework developed here suggests several natural extensions.

\paragraph{Spectral density parameters.}
The proof of Brouwer's toughness conjecture~\cite{gu21} means that spectral quantities---the eigenratio $d/\lambda$, the algebraic connectivity $\mu_2$, or the normalized Laplacian gap---can serve as density parameters~$\mu$ in our framework. Since these are polynomial-time computable, the resulting thresholds would yield efficiently verifiable sufficient conditions for all the vulnerability properties treated here. The spectral bounds of Chen and Li~\cite{chen24} suggest that such an approach would be fruitful.

\paragraph{Rupture degree and tenacity.}
The rupture degree~\cite{li05} $r(G) = \max\{\omega(G - S) - |S| - m(G - S)\}$ and tenacity~\cite{cozzens95} $T(G) = \min\{(|S| + m(G - S))/\omega(G - S)\}$ incorporate all three fundamental quantities: cost $|S|$, component count $\omega(G-S)$, and largest component order $m(G-S)$. They can be captured by extending the property function to two variables: $\psi_G(x, m) = \min\{|S| : \omega(G - S) = x,\; m(G-S) = m\}$, with the density function and threshold machinery carrying over naturally.

\paragraph{Edge-removal analogues.}
The edge-removal property function $\psi'_G(x) = \min\{|F| : F \subseteq E(G),\; \omega(G - F) = x\}$ is well-defined and supports a parallel framework. Edge toughness equals the strength of the graph, which is polynomial-time computable via matroid optimization. Edge integrity~\cite{bagga94} and component-order edge connectivity have been studied independently; our framework would unify their threshold analysis.

\paragraph{Directed and hypergraph extensions.}
Bang-Jensen et~al.~\cite{bangjensen22} initiated the study of component-order connectivity in directed graphs. Extending the property function to directed graphs---using strongly connected components---is natural but largely unexplored. Hypergraph extensions face the additional modeling choice between weak and strong vertex deletion.

\paragraph{Random graph thresholds.}
For $G(n,p)$, the deterministic $\delta$-thresholds derived here correspond, via the minimum-degree--edge-probability relationship $\delta \approx np$, to phase transitions in the random graph model. Making this connection precise---especially for $t$-toughness and $(k, \ell)$-component-order-connectivity---would connect our deterministic framework to probabilistic phenomena.

\end{document}